\documentclass[12pt]{amsart}
\usepackage{SSdefn}
\newcommand{\pe}{\mathfrak{pe}}
\newcommand{\Pe}{\mathbf{Pe}}
\newcommand{\gl}{\mathfrak{gl}}
\newcommand{\gen}{\mathrm{gen}}

\newcommand{\SVec}{\mathrm{SVec}}

\author{Rohit Nagpal}
\address{Department of Mathematics, The University of Chicago, Chicago, IL}
\curraddr{Department of Mathematics, University of Michigan, Ann Arbor, MI}
\email{\href{mailto:rohitna@umich.edu}{rohitna@umich.edu}}
\urladdr{\url{http://www-personal.umich.edu/~rohitna/}}

\author{Steven V Sam}
\address{Department of Mathematics, University of Wisconsin, Madison, WI}
\curraddr{Department of Mathematics, University of California, San Diego, CA}
\email{\href{mailto:ssam@math.ucsd.edu}{ssam@math.ucsd.edu}}
\urladdr{\url{http://math.ucsd.edu/~ssam/}}

\author{Andrew Snowden}
\address{Department of Mathematics, University of Michigan, Ann Arbor, MI}
\email{\href{mailto:asnowden@umich.edu}{asnowden@umich.edu}}
\urladdr{\url{http://www-personal.umich.edu/~asnowden/}}
\thanks{SS was partially supported by NSF grant DMS-1500069. AS was partially supported by NSF grants DMS-1303082 and DMS-1453893.}

\subjclass[2010]{%
13E05, 
13A50.
}

\date{September 29, 2018}

\title[Noetherianity of some degree two skew-tca's]{Noetherianity of some degree two\\ twisted skew-commutative algebras}

\begin{document}

\begin{abstract}
A major open problem in the theory of twisted commutative algebras (tca's) is proving noetherianity of finitely generated tca's. For bounded tca's this is easy; in the unbounded case, noetherianity is only known for $\Sym(\Sym^2(\bC^{\infty}))$ and $\Sym(\lw^2(\bC^{\infty}))$. In this paper, we establish noetherianity for the skew-commutative versions of these two algebras, namely $\lw(\Sym^2(\bC^{\infty}))$ and $\lw(\lw^2(\bC^{\infty}))$. The result depends on work of Serganova on the representation theory of the infinite periplectic Lie superalgebra, and has found application in the work of Miller--Wilson on ``secondary representation stability'' in the cohomology of configuration spaces.
\end{abstract}

\maketitle
\tableofcontents

\section{Introduction}

\subsection{Statement of results}

This paper is a sequel to \cite{sym2noeth}. Recall that a {\bf twisted commutative algebra} (tca) is a commutative $\bC$-algebra equipped with an action of the infinite general linear group $\GL_{\infty}$ by algebra homomorphisms under which it forms a polynomial representation. A major open problem in tca theory is proving noetherianity of finitely generated tca's. For so-called bounded tca's, this is straightforward \cite[Prop.~9.1.6]{expos}. The main result of \cite{sym2noeth} states that the unbounded tca's $\Sym(\Sym^2(\bC^{\infty}))$ and $\Sym(\lw^2(\bC^{\infty}))$ are noetherian. Currently, these are the only known examples of noetherianity for unbounded tca's.

One can also consider skew-commutative analogues of tca's, the typical examples being exterior (rather than symmetric) algebras. The main theorem of this paper is a skew analogue of \cite{sym2noeth} (see \S \ref{s:prelim} for the precise definitions of the terms):

\begin{theorem} \label{mainthm}
The twisted skew-commutative algebras $\lw(\Sym^2(\bC^{\infty}))$ and $\lw(\lw^2(\bC^{\infty}))$ are noetherian.
\end{theorem}

\begin{remark} \label{rmk1}
The two algebras are ``transposes'' of each other (see Remark~\ref{rmk:transpose}), and so the noetherianity of one of them implies it for the other. For this reason, we work exclusively with $\lw(\Sym^2(\bC^{\infty}))$ in this paper.
\end{remark}

\subsection{Idea of proof}

The proof of Theorem~\ref{mainthm} follows the proof of \cite{sym2noeth} closely, so we start by recalling how it goes. Let $B=\Sym(\Sym^2(\bC^{\infty}))$, and let $\Mod_B$ denote the category of $B$-modules. Define $\Mod_B^{\gen}$ (the ``generic category'') to be the quotient of $\Mod_B$ by the Serre subcategory $\Mod_B^{\tors}$ of modules with proper support (i.e., every element has non-zero annihilator). The approach of \cite{sym2noeth} is to understand the categories $\Mod_B^{\tors}$ and $\Mod_B^{\gen}$ separately, and then understand something about how they glue together to form $\Mod_B$, and finally use all of this to deduce the noetherianity result. We pursue a similar approach to prove Theorem~\ref{mainthm}. The main conceptual difficulty (at least for us) is carrying out the analysis of the generic category, so we focus on that here.

We start by recalling the analysis of $\Mod_B^{\gen}$. Geometrically, $\Spec(B)$ is the space of symmetric bilinear forms on $\bC^{\infty}$. An object of $\Mod_B$ is a $\GL_{\infty}$-equivariant quasi-coherent sheaf on this space, and an object of $\Mod_B^{\gen}$ is an equivariant sheaf on the ``open orbit'' of non-degenerate forms. (There is not literally such an open orbit, but this is a useful picture to have in mind.) Since the stabilizer of a non-degenerate form is the infinite orthogonal group $\bO_{\infty}$, we expect an equivalence between $\Mod_B^{\gen}$ and some category of representations of $\bO_{\infty}$. In fact, we show
\begin{equation} \label{eq:equiv}
\Mod_B^{\gen} = \Rep(\bO_{\infty}),
\end{equation}
where the right side is the category of algebraic representations of the infinite orthogonal group $\bO_{\infty}$ as studied in \cite{infrank}. To be a little more precise, we fix a non-degenerate symmetric bilinear form $\Sym^2(\bC^{\infty}) \to \bC$. This gives us an $\bO_{\infty}$-equivariant map of algebras $B \to \bC$. Base change under this map defines a functor $\Mod_B \to \Rep(\bO_{\infty})$ which induces the equivalence \eqref{eq:equiv}. The theory of algebraic representations of $\bO_{\infty}$ is well-understood, and so this equivalence tells us all we need to know about $\Mod_B^{\gen}$. 

We now explain the analogue of the above picture in the present setting. Initially, it is not clear how one should proceed: every positive degree element of $\lw(\Sym^2(\bC^{\infty}))$ is nilpotent, so there is not a geometric picture to work with, and thus not even a clear guess for how to describe the generic category. Our main insight is that by systematically working with super objects these difficulties disappear. First, we note that there is little difference between the polynomial representation theories of $\GL_{\infty}$ and $\GL_{\infty \mid \infty}$, and so it suffices to prove noetherianity of the ``twisted super skew-commutative algebra'' $\lw(\Sym^2(\bC^{\infty \mid \infty}))$. Next, we note that there is little difference between $\lw(\Sym^2(\bC^{\infty \mid \infty}))$ and $A=\Sym(\Sym^2(\bC^{\infty \mid \infty})[1])$, where $[1]$ denotes shift in super degree, and so it suffices to prove noetherianity of $A$, which is (super) commutative. We are now in a situation reminiscent of \cite{sym2noeth}: $\Spec(A)$ is the space of periplectic forms on $\bC^{\infty \mid \infty}$, and so we expect an equivalence
\begin{equation} \label{eq:equiv2}
\Mod_A^{\gen} = \Rep(\Pe_{\infty}),
\end{equation}
where $\Rep(\Pe_{\infty})$ is the category of algebraic representations of the infinite periplectic supergroup. (It is easier to work with the Lie superalgebra $\fpe_\infty$, so we will do that instead.) More explicitly, by fixing a non-degenerate periplectic form $\Sym^2(\bC^{\infty \mid \infty})[1] \to \bC$, we obtain a $\fpe$-equivariant algebra homomorphism $A \to \bC$. Base change along this map defines a functor $\Mod_A \to \Rep(\fpe_{\infty})$, and we show that this induces an equivalence as in \eqref{eq:equiv2}. The algebraic representation theory of $\fpe_{\infty}$ has been worked out by Serganova \cite{serganova}, and is quite similar to the theory for $\bO_{\infty}$. Thus \eqref{eq:equiv2} supplies us with all the information we need about $\Mod_A^{\gen}$.

\subsection{Motivation}

$A$-modules appear in recent work of Miller--Wilson \cite{MW} on ``secondary representation stability'' of the rational homology of connected non-compact manifolds (of finite type and dimension $\ge 2$). More specifically, work of Church--Ellenberg--Farb \cite{CEF} shows that each homology group of such a manifold is finitely generated as an ``FI-module'' (rationally, FI-modules are equivalent to modules over the tca $\Sym(\bV)$, see \cite[Proposition 1.3.5]{symc1}), and secondary stability can be phrased as saying that the set of minimal generators of these homology groups (as the homological degree varies) carries an $A$-module structure, and are graded in such a way that each graded piece is a finitely generated $A$-module. Their proof crucially depends upon Theorem~\ref{mainthm}.

Another motivation for our work comes from Koszul duality. Given a finitely generated $\Sym(\Sym^2 \bV)$-module $M$, the results of our previous paper \cite{sym2noeth} shows that $M$ has a finitely generated free resolution. Standard properties of Koszul duality imply that the space of minimal generators of the resolution can be given the structure of an $A$-module, and in fact, it is a direct sum of its ``linear strands''. In \cite[\S 6]{symc1}, we studied the $\Sym(\bV[1])$-module structure provided by Koszul duality for finitely generated $\Sym(\bV)$-modules and used to construct an interesting auto-equivalence on the derived category of finitely generated $\Sym(\bV)$-modules. The noetherianity result proved here is a starting point for the Koszul duality between $\Sym(\Sym^2 \bV)$-modules and $A$-modules. One subtlety is that in general, the number of linear strands need not be finite (i.e., Castelnuovo--Mumford regularity need not be finite), in contrast with the case of $\Sym(\bV)$-modules.

\subsection{Outline}

In \S \ref{s:prelim} we recall some background material on tca's and their super analogs. In \S \ref{sec:Perep} we introduce the category $\Rep(\fpe)$ of algebraic representations of the infinite periplectic Lie superalgebra, and recall Serganova's work on this category. In \S \ref{ss:subgroups}, we analyze some subgroups of $\GL(\bV)$ that are needed in the subsequent sections. In \S \ref{sec:serre-quotient} we introduce the notion of a torsion $A$-module, and define the Serre quotient category $\Mod_K=\Mod_A/\Mod_A^{\tors}$. In \S \ref{s:phi} and \S \ref{s:equiv}, we show that these two categories are equivalent; this is where the meat of the paper lies. Finally, in \S \ref{s:proof} we prove Theorem~\ref{mainthm}. 

\subsection{Notation}

We now fix some notation that will be in effect for the entire paper.
\begin{itemize}
\item For a super vector space $V$, we write ${}_0V$ and ${}_1V$ for the graded pieces of $V$. We refer to this as the {\bf super grading}. Most super vector spaces we consider will be endowed with an additional grading (indexed by $\bZ$ or $\bZ/2$), compatible with the super grading, called the {\bf central grading}. We write $V_n$ for the central degree $n$ piece. We write $(-)[1]$ for shift in super grading, so that ${}_0(V[1])={}_1V$ and ${}_1(V[1])={}_0V$.

\item  We let $\bV$ be the super vector space $\bC^{\infty \vert \infty}=\bigcup_{n \ge 0} \bC^{n \mid n}$. We let $e_1, e_2, \ldots$ be a basis for the even part ${}_0\bV$ and let $f_1, f_2, \ldots$ be a basis for the odd part ${}_1\bV$. We let $\GL_{\infty\mid \infty}=\bigcup_{n \ge 0} \GL_{n \mid n}$. We think of this as acting on $\bV$.

\item We let $\epsilon$ be a basis vector for $\bC[1]$, and write $\epsilon^n$ for the resulting basis vector of $\bC[n]=\bC[1]^{\otimes n}$.

\item Let $x_{i,j}=e_i f_j \epsilon$, let $y_{i,j}=e_i e_j \epsilon$, and let $z_{i,j}=f_i f_j \epsilon$, regarded as elements of $\Sym^2(\bV)[1]$. Note that $x_{i,j}$ has super degree~0 while $y_{i,j}$ and $z_{i,j}$ have super degree~1. These form a basis of $\Sym^2(\bV)[1]$, assuming one takes into account the identifications $y_{i,j}=y_{j,i}$ and $z_{i,j}=-z_{j,i}$.

\item We let $A$ be the symmetric algebra $\Sym(\Sym^2(\bV)[1])$. This is the (super) polynomial ring in the variables $x_{i,j}$, $y_{i,j}$, and $z_{i,j}$. In particular, the variables $x_{i,j}$ are in the center of $A$ and  we have the following relations \begin{align*} 
y_{i,j} y_{k,l} &= - y_{k, l} y_{i,j},\\
z_{i,j} z_{k,l} &= - z_{k, l} z_{i,j},\\
y_{i,j} z_{k,l} &= - z_{k, l} y_{i,j}.
\end{align*}
As explained in \S \ref{s:prelim}, we regard $A$ as a super tca. 

\item Let $\omega \colon \Sym^2(\bV)[1] \to \bC$ be the linear map defined by $\omega(x_{i,j})=\delta_{i,j}$ and $\omega(y_{i,j})=\omega(z_{i,j})=0$. This is an odd symmetric form on $\bV$. We let $\Pe \subset \GL(\bV)$ be the stabilizer of $\omega$, the infinite periplectic group. The Lie superalgebra of $\Pe$ is $\fpe$.

\item We let $\fm$ be the ideal of $A$ generated by the following elements: (i) the $x_{i,j}$ with $i \ne j$; (ii) the $x_{i,i}-1$; (iii) the $y_{i,j}$; and (iv) the $z_{i,j}$. Of course, $\fm$ is just the kernel of the algebra homomorphism $A \to \bC$ induced by $\omega$, and is therefore $\fpe$-stable. (Note: $\fm$ is \emph{not} $\GL_{\infty\mid\infty}$ stable. In the notation and terminology introduced in \S \ref{s:prelim}, we should really say that $\fm$ is an ideal of $\vert A \vert$.) We let $S$ be the set of super homogeneous elements of $A$ not belonging to $\fm$ (they all have super degree~0). This is a multiplicative subset of $A$.

\item Greek letters such as $\lambda, \mu, \nu, \dots$ will often denote integer partitions, which are finite weakly decreasing sequences of non-negative integers. These are used to index Schur functors $\bS_\lambda$. We will identify integer partitions with Young diagrams. We will denote the sum of parts of a partition $\lambda$ by $\vert \lambda \vert$. The notation $n \times k$ will denote the partition $(k, k, \dots, k)$ with $k$ repeated $n$ times, and $\emptyset$ denotes the unique partition of $0$.
\end{itemize}

\section{Preliminaries} \label{s:prelim}

\subsection{Polynomial representations of $\GL_{\infty}$ and tca's}

In this section, we recall some background material. We refer to \cite{expos} for more details. Let $\GL_{\infty}$ be the group $\bigcup_{n \ge 1} \GL_n$ and let $\bC^{\infty}=\bigcup_{n \ge 1} \bC^n$. A representation of $\GL_{\infty}$ is {\bf polynomial} if it decomposes as a (perhaps infinite) direct sum of Schur functors $\bS_{\lambda}(\bC^{\infty})$. We let $\cV^{\circ}$ denote the category of such representations. It is a semi-simple abelian category. Furthermore, it is closed under tensor product. A {\bf twisted commutative algebra} (tca) is a commutative algebra object in this tensor category. Concretely, a tca is a commutative associative unital $\bC$-algebra $B$ equipped with an action of $\GL_{\infty}$ by algebra homomorphisms under which, as a linear representation, it is polynomial. We write $\vert B \vert$ when we want to refer to the algebra $B$ without thinking of it as a tca.

Let $B$ be a tca. By a {\bf $B$-module} we will always mean a module object in $\cV^{\circ}$. (We use the term $\vert B \vert$-module for a module over the algebra $\vert B \vert$ with no extra structure.) Concretely, a $B$-module is a $\GL_{\infty}$-equivariant module over $\vert B \vert$ which, as a linear representation, is polynomial. There is an obvious notion of finite generation for modules. We say that $B$ is {\bf noetherian} if every submodule of a finitely generated $B$-module is again finitely generated.

\begin{remark}
If $B$ is noetherian then every ideal of $B$ is finitely generated. It is unknown if the converse of this statement holds. We note that it is relatively easy to classify all of the ideals of $B=\Sym(\Sym^2(\bC^{\infty}))$ and prove their finite generation ``by hand,'' but the proof of noetherianity of $B$ (at least the one from \cite{sym2noeth}) is much more involved.
\end{remark}

In this paper we will primarily be concerned with the algebra $\lw(\Sym^2(\bC^{\infty}))$. This is an algebra object in $\cV^{\circ}$, but it is not commutative, so it is not a tca. However, it is quite close to being one. We define the notion of module and noetherianity just as for tca's.

\begin{remark} \label{rmk:transpose}
  The category $\cV^{\circ}$ admits a transpose functor $(-)^\dag$ (see \cite[\S 7.4]{expos} for a discussion). On simple objects, it is given by $\bS_{\lambda}(\bC^{\infty})^{\dag} = \bS_{\lambda^{\dag}}(\bC^{\infty})$, where $\lambda^{\dag}$ is the transposed partition. The transpose functor is a tensor functor, but not a symmetric tensor functor: it interchanges the natural symmetry of the tensor functor with the graded symmetry. In fact, the transpose functor is induced by precomposing  with the functor $V \mapsto V[1]$ which shifts the super degree by $1$. It is clear that  $ \bS_{\lambda}(\bC^{\infty}[1]) = \bS_{\lambda^{\dag}}(\bC^{\infty})$ up to a possible shift in super degree which $(-)^\dag$ takes into account. From this description, one can see that $\lw(\Sym^2(\bC^{\infty}))^{\dag} = \lw(\Sym^2(\bC^{\infty}[1])) = \lw(\lw^2(\bC^{\infty}))$. And so, as stated in Remark~\ref{rmk1}, it suffices to prove the main theorem for $\lw(\Sym^2(\bC^{\infty}))$. 
\end{remark}

\subsection{Polynomial representations of $\GL_{\infty \mid \infty}$ and super tca's}

A {\bf polynomial representation} of $\GL_{\infty \mid \infty}$ is one that decomposes as a direct sum of $\bS_{\lambda}(\bV)$'s and $\bS_{\lambda}(\bV)[1]$'s. We let $\cV$ be the category of such representations. We let $\cV_0$ be the subcategory of representations that decompose as a direct sum of just $\bS_{\lambda}(\bV)$'s. These are both semi-simple abelian categories and closed under tensor product. (To see this, it suffices to note that $\bS_{\lambda}(\bV)$ is an irreducible representation of $\GL_{\infty \mid \infty}$. This can be deduced from the fact that $\bS_\lambda(\bC^{n|m})$ is an irreducible representation of $\GL_{n|m}$ for all $n,m$, which follows from the discussion in \cite[\S 3.2]{chengwang}.) We can consider algebra objects in this category, the commutative ones being super analogues of tca's, and modules for them. We define noetherianity as for tca's.

The category $\cV^{\circ}$ can equivalently be thought of as the category of Schur functors, and one can evaluate a Schur functor on an object of any symmetric tensor category. We therefore have a functor
\begin{displaymath}
\cV^{\circ} \to \cV_0, \qquad \bS_{\lambda}(\bC^{\infty}) \mapsto \bS_{\lambda}(\bV).
\end{displaymath}
This is easily seen to be an equivalence of abelian tensor categories. It follows that an algebra object in $\cV^{\circ}$ is noetherian if and only if the corresponding object in $\cV_0$ is. Thus to prove Theorem~\ref{mainthm}, it suffices to show that $\lw(\Sym^2(\bV))$ is noetherian, as an algebra in $\cV_0$ or $\cV$. (Proving the result in $\cV$ is a priori stronger, as it allows for more modules, but is easily seen to be equivalent.)

Every object of $\cV$ is a super vector space, and therefore has a super grading. Every object of $\cV$ also admits a central grading from the action of the ``center'' of $\GL_{\infty \mid \infty}$. (This group does not contain the scalar matrices, so does not actually have a center. However, for any given element of a representation one can mimic the action of what should be the center by taking a matrix that is approximately scalar. See \cite[\S 2.2.2]{infrank} for details.) Explicitly, the simple objects $\bS_{\lambda}(\bV)$ and $\bS_{\lambda}(\bV)[1]$ are concentrated in central degree $\vert \lambda \vert$.

Recall that $A=\Sym(\Sym^2(\bV)[1])$. This is an algebra object in $\cV$. Let $A'=\lw(\Sym^2(\bV))$; this is also an algebra in $\cV$.

\begin{proposition}
The module categories $\Mod_A$ and $\Mod_{A'}$ are equivalent. In particular, $A$ is noetherian if and only if $A'$ is.
\end{proposition}

\begin{proof}
Since $A$ is concentrated in even central degrees, any $A$-module decomposes, as an $A$-module, as the direct sum of its even and odd central degree pieces. The same is true for $A'$. We first show that the two categories of modules concentrated in even central degrees are equivalent.

Let $T(\bC[1])$ denote the tensor algebra on $\bC[1]$. Recall that $\epsilon$ is a basis vector of $\bC[1]$. We first observe that $A$ can be identified with the subalgebra $\bigoplus_{n \ge 0} A'_{2n} \epsilon^n$ of $A' \otimes T(\bC[1])$ via $A_{2n} = \Sym^n(\Sym^2(\bV)[1]) = \bigwedge^n(\Sym^2(\bV))[n] = A'_{2n} \epsilon^n$. Now let $M'$ be an $A'$-module concentrated in even central degrees. Put
\begin{displaymath}
M = \bigoplus_{n \ge 0} M'_{2n} \epsilon^n \subset M' \otimes T(\bC[1]).
\end{displaymath}
The ambient space $M' \otimes T(\bC[1])$ is an $A' \otimes T(\bC[1])$ module, and one readily verifies that $M$ is an $A$-submodule. The construction $M' \mapsto M$ is reversible, with exactly the same construction for the reverse. This gives the desired equivalence.

The equivalence for modules in odd central degrees is similar. If $M'$ is such a module, then
\begin{displaymath}
M = \bigoplus_{n \ge 0} M'_{2n+1} \epsilon^n
\end{displaymath}
is an $A$-module, and $M' \mapsto M$ is the equivalence.
\end{proof}

\subsection{A result about $\fm$} \label{sec:fm}

Let $Q_1$ be the set of partitions so that for each box in the main diagonal, the number of boxes in the same row and to the right of it is exactly $1$ more than the number of boxes in the same column and below it. By \cite[Ex. I.8.6(d)]{macdonald}, we have the following decomposition:
\[
A = \bigoplus_{\lambda \in Q_1} \bS_\lambda(\bV)[|\lambda|/2]
\] 
where the shifts are in superdegree. Let $\fp_n \subset A$ be the ideal generated by $\bS_{n \times (n+1)}(\bV)[n(n+1)/2]$.

\begin{lemma} \label{lem:pn-gen}
$\prod_{1 \le i \le j \le n} y_{i,j}$ generates $\fp_n$. 
\end{lemma}

\begin{proof}
Let $N=\binom{n+1}{2}$. We have
\begin{displaymath}
A_{n(n+1)}(\bC^{n\mid 0}) =  \lw^N(\Sym^2(\bC^n))[N] = \bS_{n \times (n+1)}(\bC^n)[N].
\end{displaymath}
This is $1$-dimensional and spanned by the element under discussion, so $\prod_{1 \le i \le j \le n} y_{i,j} \in \fp_n$. Since $\fp_n$ is generated by the Schur functor $\bS_{n \times (n+1)}[N]$, we conclude that $\fp_n$ is generated by this element.
\end{proof}

\begin{proposition} \label{prop:unitideal}
We have $\fm + \fp_n = A$ for all $n \ge 1$.
\end{proposition}

\begin{proof}
Let $X_{i,j} \in \fgl_{\infty}$ be the element sending $e_i$ to $f_j$ and killing the $e_k$ with $k \ne i$ and the $f_{\ell}$. Consider the element
\[
\bv = X_{1,1} X_{1,2} X_{2,2} \cdots X_{1,n} \cdots X_{n-1,n} X_{n,n} \prod_{1 \le i \le j \le n} y_{i,j}.
\]
Expanding this, we find a term of the form
\[
\bv_0 = \pm \prod_{1 \le i \le j \le n} x_{j,j},
\]
and all other terms have the property that they contain a factor of the form $x_{i,j}$ where $i \ne j$, $y_{i,j}$, or $z_{i,j}$. Thus $\bv \equiv \bv_0 \equiv \pm 1 \pmod \fm$. Since $\fp_n$ is closed under $\fgl_{\infty}$, we have $\bv \in \fp_n$. Since $\pm \bv-1 \in \fm$, it follows that $1 \in \fm+\fp_n$.
\end{proof}

For $n >0$, let $y(n) = \prod_{1 \le i \le j \le n} y_{i,j}$.

\begin{corollary} \label{cor:ann}
If $\fa$ is a non-zero ideal of $A$, then $\fa \supseteq \fp_n$ for some $n$. In particular, $\fa + \fm = A$.
\end{corollary}

\begin{proof}
Pick $n$ so that $\fa(\bC^n) \ne 0$. Then $\fa(\bC^n)$ is a nonzero homogeneous ideal in the exterior algebra $\bigwedge(\Sym^2(\bC^n))$. In particular, it contains its top degree piece $\bigwedge^{n(n+1)/2} (\Sym^2(\bC^n))$ which is spanned by $y(n)$, so by Lemma~\ref{lem:pn-gen}, $\fa \supseteq \fp_n$. Now use Proposition~\ref{prop:unitideal}.
\end{proof}

\subsection{More on ideals}

The subalgebra of $A$ generated by the $x_{i,j}$ is the commutative algebra $\Sym({}_0\bV \otimes {}_1\bV)$. Given a partition $\lambda$, let $x_\lambda$ be a nonzero vector in $\bS_\lambda({}_0\bV) \otimes \bS_\lambda({}_1\bV)$ which is a highest weight vector with respect to the upper triangular matrices in $\fgl({}_0 \bV) \times \fgl({}_1 \bV)$. 

If $\ell(\lambda) \le n$, write $\lambda\{n\}$ for the partition $(\lambda_1 + n+1, \dots, \lambda_n + n+1, \lambda^\dagger)$.

\begin{lemma} \label{lem:lambda-hw}
The $\fgl(\bV)$-subrepresentation of $A$ generated by $y(n) x_\lambda$ is $\bS_{\lambda\{n\}}(\bV)$.
\end{lemma}

\begin{proof}
Replace $\bV$ with $\bC^{n|n}$. Consider the upper-triangular matrices in $\fgl_{n|n}$ where we have ordered the even variables before the odd ones. We claim that $y(n) x_\lambda$ is a highest weight vector for this choice of Borel. Both $y(n)$ and $x_\lambda$ are eigenvectors for the upper-triangular matrices in $\fgl_n \times \fgl_n$, so the same is true for $y(n) x_\lambda$. The remainder of the Borel is the upper-right block, which consists of maps from ${}_1\bV$ to ${}_0\bV$. However, the action of any such matrix replaces $x_{i,j}$ with some $y_{k,\ell}$; since $y(n)$ contains the product of all of the $y$ variables, the action is $0$ on $y(n) x_\lambda$, so we conclude it is a highest weight vector. The even part of its weight is $(\lambda_1 + n+1, \dots, \lambda_n + n+1)$ and the odd part of its weight is $\lambda$, so we conclude that it generates the Schur functor $\bS_{\lambda\{n\}}(\bV)$ (see \cite[\S 3.2.2]{chengwang} for this last statement).
\end{proof}

\begin{corollary} \label{cor:ess-bound}
Suppose $\lambda \in Q_1$. If $n \times (n+1) \subseteq \lambda$ then $\fp_n$ contains $\bS_\lambda(\bV)$. 
\end{corollary}

\begin{proof}
If $n \times (n+1) \subseteq \lambda$, then $\lambda = \mu\{n'\}$ for some $n' \ge n$ and some partition $\mu$. From Lemma~\ref{lem:lambda-hw}, we see that $\bS_{n' \times (n'+1)}(\bV)$ generates $\bS_\lambda(\bV)$, and from Lemma~\ref{lem:pn-gen}, we see that $\bS_{n \times (n+1)}(\bV)$ generates $\bS_{n' \times (n'+1)}(\bV)$.
\end{proof}

\subsection{The Borel subgroup and the maximal torus} \label{ss:subgroups}

Ordering our basis of $\bV$ as $e_1, e_2, \ldots,$ and $f_1, f_2, \ldots$, we can think of elements of $\GL(\bV)$ as block matrices
\begin{displaymath}
\begin{pmatrix}
a & b \\ c & d
\end{pmatrix}.
\end{displaymath}
Let $B \subset \GL(\bV)$ be the subgroup where $a$, $c$, and $d$ are upper-triangular, and $b$ is strictly upper triangular. The determinant of such a matrix is simply the product of the determinants of $a$ and $d$. Let $\fb$ be the Lie algebra of $B$. Note that $B$ is a Borel subgroup since it is the subgroup of upper-triangular matrices with respect to the ordering $f_1 < e_1 < f_2 < e_2 < \cdots$.

Let $\bG_m$ denote the multiplicative group, and let $T=\bG_m^{\infty}$ where all but finitely many coordinates are~1. We denote elements of $T$ as $(\alpha_1, \alpha_2, \ldots)$. We regard $T$ as a subgroup of $\GL(\bV)$ by
\begin{displaymath}
\alpha \mapsto \begin{pmatrix} \alpha & 0 \\ 0 & \alpha^{-1} \end{pmatrix}.
\end{displaymath}
In other words, $\alpha \cdot e_i=\alpha_i e_i$ and $\alpha \cdot f_i=\alpha_i^{-1} f_i$. This $T$ is the maximal torus of $\Pe$, and is the intersection of $B$ and $\Pe$.

\begin{lemma} \label{lem:iwasawa}
$\fb + \fpe = \fgl$ and $\fb \cap \fpe$ is the Lie algebra of $T$.
\end{lemma}

\begin{proof}
It suffices to prove that $\fb_n + \fpe_n = \fgl_{n|n}$ for all $n$ where $\fgl_{n|n} = \End(\bC^{n|n})$ and $\fb_n$ and $\fpe_n$ are subalgebras of $\fgl_{n|n}$ defined in an analogous way as $\fb$ and $\fpe$. We can do this by a dimension count. First, we remark that $\fpe_n$ consists of matrices of the form $\begin{pmatrix} a & b \\ c & -a^T \end{pmatrix}$ where $a,b,c$ are $n \times n$ matrices with $b=b^T$ and $c=-c^T$ \cite[\S 1.1.5, equation (1.14)]{chengwang}. So $\fb_n \cap \fpe_n$ consists of matrices of the form $\begin{pmatrix} a & 0 \\ 0 & -a \end{pmatrix}$ where $a$ is a diagonal matrix. We see that 
\[
\dim(\fpe_n + \fb_n) = \dim(\fpe_n) + \dim(\fb_n) - n = 2n^2 + (2n^2+n) - n = 4n^2 = \dim(\fgl_{n|n}),
\]
so $\fpe_n + \fb_n = \fgl_{n|n}$.
\end{proof}

\begin{corollary} \label{cor:pe-gens}
Let $V$ be a $\fgl$-representation. Let $\{x_i\}$ be a complete set of highest weight vectors for $V$ with respect to the Borel subalgebra $\fb$. Then $\{x_i\}$ generate $V$ as a $\fpe$-representation.
\end{corollary}

\begin{proof}
Every $v \in V$ is a linear combination of $a_1 \cdots a_r x_i$ where $a_i \in \fgl$. By induction on $r$, we will show that this belongs to $\cU(\fpe) x_i$. If $r=0$, there is nothing to show; if $r>0$, write $a_2 \cdots a_r x_i$ as a linear combination of $p_1 \cdots p_s x_i$ with $p_i \in \fpe$. It suffices to show that $a_1 p_1 \cdots p_s x_i \in \cU(\fpe) x_i$, which we will do by induction on $s$. Write $a_1 = b + p$ where $b \in \fb$ and $p \in \fpe$. If $s=0$, then we have $a_1 x_i = bx_i + px_i$; the second term is in $\cU(\fpe) x_i$ by definition, and the first term is a scalar multiple of $x_i$ since it is a highest weight vector, so $a_1 x_i \in \cU(\fpe) x_i$. If $s>0$, write $[b,p_1] = b' + p'$ where $b' \in \fb$ and $p' \in \fpe$. Then we have
\[
bp_1 \cdots p_s x_i = p_1 bp_2 \cdots p_s x_i + b' p_2 \cdots p_s x_i + p' p_2 \cdots p_s x_i.
\]
Now the first and second terms are in $\cU(\fpe) x_i$ by induction on $s$, and the last term is in $\cU(\fpe)x_i$ by definition, so we are done.
\end{proof}

\section{Stable representation theory of the periplectic group}
\label{sec:Perep}

We say that a representation of $\fpe$ is {\bf algebraic} if it appears as a subquotient of a finite direct sum of the spaces $T_n =\bV^{\otimes n}$ and $T_n[1]$.\footnote{The ``restricted dual'' $\bV_*$ is isomorphic to $\bV$ as a representation of $\fpe$, so one does not get anything new by considering mixed tensors.} We write $\Rep(\fpe)$ for the category of algebraic representations. The category $\Rep(\fpe)$ is closed under tensor products. Serganova \cite{serganova} has determined the structure of this category, and in this section we summarize the results and recast them in the style of \cite{infrank}. We remark that one of the conclusions of \cite{serganova}, namely that $\Rep(\fpe)$ is equivalent to $\Rep(\bO)$, is incorrect, see Remark~\ref{rmk:error}.

In \cite[(4.2.5)]{infrank}, we defined the downwards Brauer category, and in \cite[(4.2.11)]{infrank} we defined a signed variant. Here we introduce a different signed variant of this category, which we simply denote by $\cC$. It is defined as follows:
\begin{itemize}
\item The objects of $\cC$ are finite sets.
\item The space of morphisms $\uHom_{\cC}(L, L')$ is the super vector space spanned by pairs $(\Gamma, f)$, where $\Gamma$ is a matching on $L$ equipped with an orientation on its edge set (i.e., a total ordering modulo the action of even permutations) and $f$ is a bijection $L \setminus V(\Gamma) \to L'$, modulo the relations $(\Gamma, f)=-(\Gamma', f)$ if $\Gamma'$ is obtained from $\Gamma$ by reversing the orientation on the edge set. The (super) degree of $(\Gamma, f)$ is the number of edges in $\Gamma$.
\item The composition of $(\Gamma, f) \colon L \to L'$ and $(\Gamma', f') \colon L' \to L''$ is $(\Gamma \cup f^{-1}(\Gamma'), f' \circ f)$, where the orientation on the edge set of $\Gamma \cup f^{-1}(\Gamma')$ is the one obtained by putting the edges of $\Gamma$ before those of $f^{-1}(\Gamma')$.
\end{itemize}
We write $\Mod_{\cC}$ for the category of enriched functors $M \colon \cC \to \SVec$.

We now define an object $\cK$ of $\Mod_{\cC}$. For a finite set $L$, we put $\cK_L=\bV^{\otimes L}$. For a morphism $(\Gamma, f) \colon L \to L'$, we define $\cK_L \to \cK_{L'}$ by applying the pairing $\omega$ to the tensor factors paired by $\Gamma$ and using $f$ on the remaining tensor factors. Each $\cK_L$ belongs to $\Rep(\fpe)$, and the maps $\cK_L \to \cK_{L'}$ are maps of $\fpe$-representations, so $\cK$ can be considered as a representation of $\cC$ in the category $\Rep(\fpe)$. We therefore obtain a functor
\begin{displaymath}
\Phi \colon \Mod_{\cC}^{\rf} \to \Rep(\fpe), \qquad \Phi(M)=\Hom_{\cC}(M, \cK)
\end{displaymath}
and a functor
\begin{displaymath}
\Psi \colon \Rep(\fpe)^{\rf} \to \Mod_{\cC}, \qquad \Psi(N)=\Hom_{\fpe}(N, \cK),
\end{displaymath}
as in \cite[(2.1.10)]{infrank}. Here $(-)^{\rf}$ denotes the full subcategory of finite length objects.

\begin{theorem} \label{thm:PeBrauer}
The functors $\Phi$ and $\Psi$ are mutually quasi-inverse contravariant equivalences between $\Mod_{\cC}^{\rf}$ and $\Rep(\fpe)^{\rf}$.
\end{theorem}

\begin{proof}
We apply the criterion of \cite[(2.1.11)]{infrank}. (We are not exactly in the situation discussed there, but the same criterion and proof still apply.) Part~(a) follows from \cite[Proposition~3(d)]{serganova} and \cite[Lemma~17]{serganova}. For Part~(b), consider a simple object $V^{\lambda}$ of $\Rep(\fpe)$, in the notation of \cite{serganova}. Suppose that the partition $\lambda$ is of size $n$. Then $\Hom_{\fpe}(V^{\lambda}, \cK_{[n]})$ is the Specht module $\bM_{\lambda}$ by \cite[Proposition~3(d)]{serganova}. Furthermore, if $m \ne n$ then $\Hom_{\fpe}(V^{\lambda}, \cK_{[m]})=0$ by \cite[Proposition~3(d)]{serganova}, since then the socle of $\cK_{[m]}$ has no copy of $V^{\lambda}$ in it.
\end{proof}

\begin{proposition} \label{prop:pe-inj}
We have the following:
\begin{enumerate}[\indent \rm (a)]
\item The $\bS_{\lambda}(\bV)$ are finite length representations of $\fpe$.
\item The $\bS_{\lambda}(\bV)$ are exactly the indecomposable injective objects of $\Rep(\fpe)^\rf$.
\item Every object of $\Rep(\fpe)^{\rf}$ has a finite length resolution by finite length injective objects $($i.e., finite sums of indecomposable injectives$)$.
\end{enumerate}
\end{proposition}

\begin{proof}  One can deduce from \cite[Lemma~17]{serganova} that the quotient of $\bS_{\lambda}(\bV)$ by its socle injects into a finite sum of $\bS_{\mu}(\bV)$'s with $\mu$ smaller than $\lambda$. An easy inductive argument using this proves (a).  (b) follows from \cite[Theorem~8]{serganova}, which states that $\bS_{\lambda}(\bV)$ is the injective envelope of the simple corresponding to $\lambda$. (c) follows from Theorem~\ref{thm:PeBrauer}, as the corresponding statement for $\Mod_{\cC}$ is clear from the simple form of $\cC$. 
\end{proof}

\begin{proposition} \label{prop:wedgesym2-diagram}
The category of finite length supermodules over $\Sym(\Sym^2( \bV)[1])$ and $\Mod_\cC^\rf$ are equivalent. In particular, it is also equivalent to  $\Rep(\fpe)^\rf$.
\end{proposition}

\begin{proof}
This follows from a signed variant of \cite[(2.4.1)]{infrank}.
\end{proof}

\begin{remark} \label{rmk:error}
Serganova claims in \cite[Theorem 9]{serganova} (without proof) that there is an equivalence between $\Rep(\fpe)$ and $\Rep(\bO)$ (both are categories of supermodules). In this remark, we explain that no such equivalence exists (even ignoring the tensor structure).

First, the existence of such an equivalence implies that the subcategories of finite length objects are also equivalent. By Proposition~\ref{prop:wedgesym2-diagram}, $\Rep(\fpe)^\rf$ is equivalent to the category of finite length supermodules over $\bigwedge(\Sym^2 \bV)$ and by \cite[Theorem 4.3.1]{infrank}, $\Rep(\bO)^\rf$ is equivalent to the category of finite length supermodules of $\Sym(\Sym^2 \bV)$. So in our language, this would be an equivalence between the categories of finite length supermodules of the tca $\Sym(\Sym^2 \bV)$ and the skew tca $\bigwedge(\Sym^2 \bV)$. Using the Koszul complex, we find
\begin{align*}
\ext^i_{\Sym(\Sym^2)}(\bS_\lambda, \bS_\mu) &\cong \hom_{\GL}(\bS_\mu, \bS_\lambda \otimes \bigwedge^i(\Sym^2))\\
\ext^i_{\bigwedge(\Sym^2)}(\bS_\lambda, \bS_\mu) &\cong \hom_{\GL}(\bS_\mu, \bS_\lambda \otimes \Sym^i(\Sym^2)).
\end{align*}
We claim that any equivalence has to send $\bS_\emptyset$ either to itself or its parity shift $\bS_\emptyset[1]$. First, if $\mu$ is a partition with the property that $\ext^\bullet_{\Sym(\Sym^2)}(\bS_\lambda, \bS_\mu) = 0$ for all $\lambda$, then $\mu$ is a single column partition of the form $(1^d)$ for some $d \ge 0$. Hence any simple solution $M$ to $\ext^\bullet_{\Sym(\Sym^2)}(\bS_\lambda, M) = 0$ for all $\bS_\lambda$ must be either $\bS_{(1^d)}$ or $\bS_{(1^d)}[1]$. If $d>0$, then there are two solutions to $\ext^2_{\Sym(\Sym^2)}(\bS_\lambda, \bS_{1^d}) \ne 0$, namely $\lambda \in \{(3,1^{d-1}), (2,1^d)\}$, and if $d=0$ there is only one solution. Similarly, any simple solution $M$ to $\ext^\bullet_{\bigwedge(\Sym^2)}(\bS_\lambda, M) = 0$ for all $\lambda$ must be $\bS_d$ or $\bS_d[1]$ for some $d \ge 0$. If $d>0$, then $\ext^2_{\bigwedge(\Sym^2)}(\bS_\lambda, \bS_d) \ne 0$ has two solutions, and if $d=0$ there is only one solution. In particular, $\bS_\emptyset$ is either sent to itself or $\bS_\emptyset[1]$. Since the parity change functor is an equivalence in Serganova's setup, we may as well compose with it if needed to assume that the proposed equivalence sends $\bS_\emptyset$ to itself.

Next, notice that there are exactly two simple solutions $M$ to $\ext^2_{\Sym(\Sym^2)}(\bS_\emptyset, M) \ne 0$, namely $\bigwedge^2(\Sym^2) \cong \bS_{2,1,1}$ or its parity shift. On the other hand, there are exactly four simple solutions $M$ to $\ext^2_{\bigwedge(\Sym^2)}(\bS_\emptyset, M) \ne 0$, namely those appearing in $\Sym^2(\Sym^2) \cong \bS_{4} \oplus \bS_{2,2}$ or their parity shifts. As we just said, the equivalence preserves $\bS_\emptyset$, and since it takes simple objects to simple objects, and preserves extension groups, we conclude that {\it no equivalence between $\Sym(\Sym^2)$ and $\bigwedge(\Sym^2)$ exists.}
\end{remark}

\section{The generic category} \label{sec:serre-quotient}

We now define a notion of ``torsion'' for $A$-modules. We begin with a variant of Nakayama's lemma. Recall that $S$ is the set of super homogeneous elements of $A \setminus \fm$.

\begin{lemma} \label{lem:nakayama}
Let $M$ be a finitely generated $A$-module such that $M=\fm M$ (with equality as $|A|$-modules). Then $S^{-1}M=0$.
\end{lemma}

\begin{proof}
Let $V \subset M$ be a finite length $\GL$-subrepresentation generating $M$ as an $A$-module. Pick $m_1, \ldots, m_k \in V$ such that the $m_i$ generate $V(\bC^{N\mid N})$ as a $\fpe_N$-representation for all $N \gg 0$ (they exist by Corollary~\ref{cor:pe-gens}). Write $m_i=\sum_i a_{i,j} n_{i,j}$ where $a_{i,j} \in \fm$ and $n_{i,j} \in M$. Let $N \gg 0$ be large enough so that the $m_i$ and the $n_{i,j}$ belong to $M'=M(\bC^{N\mid N})$ and the $a_{i,j}$ belong to $A'=A(\bC^{N \mid N})$. Let $V'=V(\bC^{N\mid N})$, let $\fm'=\fm(\bC^{N\mid N})$, and let $S'$ be the super homogeneous elements of $A'$ not in $\fm'$ (they all have super degree~0). Then $M'$ is an $A'$-module and generated (ignoring any group action) by $V'$. We have $m_i \in \fm' M'$ for all $i$, and so $gm_i \in \fm' M'$ for any $g \in \fpe_N$, since $\fm'$ is $\fpe_N$-stable. Thus $V' \subset \fm' M'$ and so $M'=\fm' M'$. Thus, by the usual version of Nakayama's lemma \cite[(4.22)]{lam}, we have $(S')^{-1} M'=0$. Therefore, for each $1 \le i \le k$ there exists $s_i \in S' \subset S$ such that $s_im_i=0$, which implies $S^{-1}M=0$.
\end{proof}

\begin{proposition} \label{prop:torsion}
Let $M$ be an $A$-module. The following conditions are equivalent:
\begin{enumerate}[\indent \rm (a)]
\item For every finitely generated submodule $M'$ of $M$ there is a non-zero ideal $\fa$ of $A$ such that $\fa M' = 0$.
\item We have $S^{-1}M=0$.
\item For every $m \in M$ there exists $a \in A$ with non-zero image in $\bC[x_{i,j}]=A/(y_{i,j},z_{i,j})$ such that $am=0$.
\end{enumerate}
\end{proposition}

\begin{proof}
Suppose (a) holds, and let us prove (b). Let $M' \subset M$ be finitely generated, and let $\fa$ be a non-zero ideal of $A$ annihilating the submodule $M'$. Since $\fa+\fm=A$ by Lemma~\ref{cor:ann}, we have $\fm M'=M'$, and so $S^{-1} M'=0$ by Lemma~\ref{lem:nakayama}. Since this holds for all finitely generated $M' \subset M$, it follows that $S^{-1} M=0$. 

Now suppose (b) holds. So given $m \in M$, there exists $s \in S$ such that $sm=0$. As $s$ has non-zero reduction in $\bC[x_{i,j}]$, one can take $a=s$ in (c). Thus (c) holds.

Finally, suppose (c) holds. Let $M'$ be a submodule of $M$ generated by $m_1, \ldots, m_k$. Let $a_i m_i=0$ with $a_i$ as in (c). Let $a=a_1 \cdots a_k$; this still has non-zero image in $A/(y_{i,j}, z_{i,j})$, since $\bC[x_{i,j}]$ is a domain, and annihilates each $m_i$. Following the proof of \cite[Prop.~2.2]{sym2noeth}, we see that there exists $n$, depending only on the $m_i$, such that $a^n(gm_i)=0$ for all $g \in \GL_{\infty\mid\infty}$. Note that $a^n \ne 0$, again since $\bC[x_{i,j}]$ is a domain. It follows that the (non-zero) ideal of $A$ generated by $a^n$ annihilates $M'$, and so (a) holds.
\end{proof}

We say that an $A$-module is {\bf torsion} if it satisfies the equivalent conditions of Proposition~\ref{prop:torsion}. We write $\Mod_A^{\tors}$ for the category of torsion modules. It is clear that this is a Serre subcategory of $\Mod_A$. We denote by $\Mod_K$ the Serre quotient $\Mod_A/\Mod_A^{\tors}$, and write $\rT \colon \Mod_A \to \Mod_K$ for the localization functor.

\section{Local structure of $A$-modules at $\fm$} \label{s:phi}

In this section, we analyze the local structure of $A$-modules at the ideal $\fm$. The main result (Proposition~\ref{prop:phiM}) shows that if $M$ is an $A$-module then the localization $S^{-1} M$ can be functorially recovered from the $\fpe$-representation $M/\fm M$. As an important corollary, we find that $S^{-1} M$ is free over $S^{-1} A$.

\subsection{Construction of $\phi$}

Following the notation from \S \ref{ss:subgroups}, let $\bC[\ol{B}]$ be the super polynomial ring in even variables $a_{i,j}$ with $i \le j$, odd variables $b_{i,j}$ with $i < j$, odd variables $c_{i,j}$ with $i \le j$ and even variables $d_{i,j}$ with $i \le j$. Then $\bC[B]$ is $\bC[\ol{B}]$ with the variables $a_{i,i}$ and $d_{i,i}$ inverted. We let $T$ act on $\bC[\ol{B}]$ as follows:
\begin{displaymath}
\alpha \cdot a_{i,j}=\alpha_i^{-1} a_{i,j}, \quad \alpha \cdot b_{i,j}=\alpha_i^{-1} b_{i,j}, \quad
\alpha \cdot c_{i,j}=\alpha_i c_{i,j}, \quad \alpha \cdot d_{i,j}=\alpha_i d_{i,j}.
\end{displaymath}

Let $V$ be a polynomial representation of $\GL(\bV)$. Then $V$ is naturally a comodule over $\bC[\ol{B}]$ (see \cite[\S 3.2]{sym2noeth}). The image of the comultiplication map $V \to V \otimes \bC[\ol{B}]$ is elementwise $T$-invariant. Let $M$ be an $A$-module. Taking $V=M$ in the previous comment and using that $\fm M$ is $T$-invariant, we thus obtain a map 
\begin{displaymath}
\phi_M \colon M \to (M/\fm M \otimes \bC[\ol{B}])^T.
\end{displaymath}
In the remainder of this section, we study this map.

\subsection{The map $\phi_A$}

We now study the map $\phi_A$:
\begin{displaymath}
\phi_A \colon A \to \bC[\ol{B}]^T = A'.
\end{displaymath}
This is an algebra homomorphism. The ring $A'$ is easy to describe:
\begin{displaymath}
A'=\bC[a_{i,j} c_{i,k}, a_{i,j}d_{i,k}, b_{i,j}c_{i,k}, b_{i,j}d_{i,k}].
\end{displaymath}
We now compute $\phi_A$ explicitly. Under comultiplication, we have
\begin{displaymath}
e_i \mapsto \sum_{k \le i} e_k a_{k,i} + \sum_{k \le i} f_k c_{k,i}, \qquad
f_i \mapsto \sum_{k\le i} e_k b_{k,i} + \sum_{k \le i} f_k d_{k,i},
\end{displaymath}
using the convention $b_{i,i}=0$. We thus have
\begin{displaymath}
x_{i,j} \mapsto \left( \sum_{k \le i} e_k a_{k,i} + \sum_{k \le i} f_k c_{k,i} \right) \cdot
\left( \sum_{\ell<j} e_{\ell} b_{\ell,j} + \sum_{\ell \le j} f_{\ell} d_{\ell,j} \right)  \cdot \epsilon
\end{displaymath}
under comultiplication. Passing to $A/\fm$, only the $e_i f_i \epsilon$ terms survive, and they all become~1. We thus find
\begin{displaymath}
\phi(x_{i,j}) = \sum_{k \le i,j} ( a_{k,i} d_{k,j} + c_{k,i} b_{k,j} ) = X_{i,j}.
\end{displaymath}
Similar computations give
\begin{align*}
\phi(y_{i,j})&=\sum_{k \le i,j} ( a_{k,i} c_{k,j}+a_{k,j} c_{k,i}) = Y_{i,j},\\
\phi(z_{i,j})&=\sum_{k \le i,j} (d_{k,i}b_{k,j}-b_{k,i}d_{k,j}) = Z_{i,j}.
\end{align*}

Define an ordering on the variables $a, b, c, d$ as follows: for $p,q \in  \{ a, b, c ,d \}$, first we define $p_{ij} > q_{k\ell}$ if $(j,i) > (\ell,k)$ in the lexicographic order, and then to compare $p_{i,j}$ and $q_{i,j}$, we use the ordering $d > c > a > b$. Extend this to order monomials using the graded lexicographic ordering. The {\bf leading term} of an element in $A'$ is the largest monomial appearing in it with nonzero coefficient. So when $i \le j$, we have
\begin{itemize}
\item the leading term of $X_{i,j}$ is $a_{i,i}d_{i,j}$, 
\item the leading term of $X_{j,i}$ is $a_{i,j} d_{i,i}$, 
\item the leading term of $Y_{i,j}$ is $a_{i,i} c_{i,j}$, and 
\item the leading term of $Z_{i,j}$ is $d_{i,i} b_{i,j}$ (note that $b_{ii} = 0$ and that $Z_{i,i}=0$ since $z_{i,i}=0$).
\end{itemize}
The leading term of a monomial in $X,Y,Z$ is the product of the corresponding leading terms. (Note: in a non-zero monomial, $Y_{i,j}$ and $Z_{i,j}$ can only appear once, since they square to zero, and so the product of leading terms is non-zero.)

\begin{proposition} \label{prop:phi-inj}
$\phi_A$ is injective.
\end{proposition}

\begin{proof}
It suffices to show that distinct monomials in the $X,Y,Z$ (where each $Y_{i,j}$ and $Z_{i,j}$ appear at most once) have distinct leading terms. If we have a product of the $a,b,c,d$ which is the leading term of some monomial in $X,Y,Z$, we just need to show that this monomial can be uniquely reconstructed. First, any instances of $c_{i,j}$ must have an accompanying $a_{i,i}$ and this corresponds to an instance of $Y_{i,j}$, and similarly for $b_{i,j}$. After removing these, we are left with a leading term in $a,d$. But again, any instance of $d_{i,j}$ with $i<j$ has an accompanying $a_{i,i}$ and this corresponds to $X_{i,j}$, and similarly for $a_{i,j}$. 
\end{proof}

\begin{lemma} \label{lem:NZD}
Let $I$ be the ideal of $\vert A \vert$ generated by elements of the form $y_{i,j}$ and $z_{i,j}$. An element $s \in A$ is a nonzerodivisor if and only if $s \notin I$.
\end{lemma}

\begin{proof} 
Clearly every element of $I$ is a zero divisor. Now suppose $s \notin I$ and pick $a \in A \setminus \{0\}$. To a monomial $m$ in variables of the form $x_{i,j}$, $y_{i,j}$ and $z_{i,j}$, we say that $\deg(m) = n$ if $n$ is the largest integer such that $m \in I^n$ (this notion of degree only satisfies $\deg(m_1 m_2) \ge \deg(m_1) + \deg(m_2)$). Since the monomials form a basis for $A$, this defines a direct sum decomposition $A = \bigoplus_{n \ge 0} A_n$ which we will use for the rest of this proof (but nowhere else in the paper). Clearly, the degree $0$ piece $m$ of $s$ is nonzero (and $m$ is a polynomial in the $x_{i,j}$ with coefficients in $\bC$, and hence a nonzerodivisor). Let $e$ be the piece of $a$ of minimal degree. Then $me$ is nonzero, and the minimal degree piece of $as$ is $me$. This shows that $as \ne 0$ and proves that $s$ is a nonzerodivisor.
\end{proof}

Recall that $S$ is the set of super degree~0 elements of $A$ not belonging to $\fm$. By Lemma~\ref{lem:NZD}, every element of $S$ is a nonzerodivisor (since $I \subset \fm$).

\begin{proposition} \label{prop:phiA}
The localization of $\phi_A$ at $S$ is an isomorphism.
\end{proposition}

\begin{proof}
Injectivity of $S^{-1} \varphi_A$ follows from Proposition~\ref{prop:phi-inj} because localization is exact. Let $x$ be a nonzero element of the form $a_{k,i} d_{k,j}$, $a_{k,i} c_{k,j}$, $b_{k,i} c_{k,j}$ or $b_{k,i} d_{k,j}$ in $S^{-1}A'$. In a similar manner as in Proposition~\ref{prop:phi-inj}, we define a total quasiorder on the variables $a, b, c, d$ as follows: for $p,q \in  \{ a, b, c ,d \}$, first we define $p_{ij} > q_{k\ell}$ if $(j,i) > (\ell,k)$ in the lexicographic order. Then we extend it to a total quasiorder on monomials using the graded lexicographic ordering. We show by induction on this quasiorder (which is clearly well-founded) that:
\begin{itemize}
\item if $x$ is of type $ad$, $bc$, $ac$, or $bd$, then it is in the image of $S^{-1} \phi_A$,
\item if $x$ is not of the form $a_{i,i} d_{i,i}$ then it is in the image of $S^{-1} \fm$,
\item if $x$ is of the form $a_{i,i} d_{i,i}$ then it is in the image of $S$ and hence is a unit in $S^{-1} A'$.
\end{itemize} 

The base case is clear because by definition $ \varphi_A({x_{1,1}}) = X_{1,1} = a_{1,1}d_{1,1}$, $Y_{1,1} = 2 a_{1,1} c_{1,1}$ are in the image and $x_{1,1} \in S$. Also note that $b_{i,i} = 0$, so $b_{1,1} c_{1,1} = b_{1,1} d_{1,1} = 0$.

For the case $k < i,j$, the following expressions and the induction hypothesis proves the hypothesis at hand: 
\begin{align*}
a_{k,i} d_{k,j} &=  (a_{k,k} d_{k,k})^{-1} (a_{k,i} d_{k,k}) (a_{k,k} d_{k,j})  \\
a_{k,i} c_{k,j} &=  (a_{k,k} d_{k,k})^{-1} (a_{k,i} d_{k,k}) (a_{k,k} c_{k,j})  \\
b_{k,i} c_{k,j} &=  (a_{k,k} d_{k,k})^{-1} (b_{k,i} d_{k,k}) (a_{k,k} c_{k,j})  \\
b_{k,i} d_{k,j} &=  (a_{k,k} d_{k,k})^{-1} (b_{k,i} d_{k,k}) (a_{k,k} d_{k,j}). 
\end{align*} 

Next we consider the case $k = i < j$. Then the equations 
\begin{align*}
X_{j,i} &= a_{i,j} d_{i,i}  + \text{lower terms of type $ad$ or $bc$}  \\
Z_{i,j} &= d_{i,i} b_{i,j}  + \text{lower terms of type $ac$ or $bd$}
\end{align*} 
imply that $ a_{i,j} d_{i,i}$ and $d_{i,i} b_{i,j}$ satisfy the hypothesis. Next the equations  \begin{align*}
b_{i,j} c_{i,i} &= (a_{i,i} d_{i,i})^{-1} (b_{i,j} d_{i,i}) (a_{i,i} c_{i,i}) \\
a_{i,j} c_{i,i} &= (a_{i,i} d_{i,i})^{-1} (a_{i,j} d_{i,i}) (a_{i,i} c_{i,i}) 
\end{align*} 
show that $b_{i,j} c_{i,i}$ and $a_{i,j} c_{i,i}$ satisfy the hypothesis. Finally, the equations \begin{align*}
X_{i,j} &= a_{i,i} d_{i,j} + c_{i,i} b_{i,j}  + \text{lower terms of type $ad$ or $bc$} \\
Y_{i,j} &= a_{i,i} c_{i,j} +  a_{i,j} c_{i,i} + \text{lower terms of type $ac$ or $bd$}  
\end{align*} show that $a_{i,i} d_{i,j}$ and $a_{i,i} c_{i,j}$ satisfy the hypothesis.

The case $k = j < i$ follows from the last case. Finally, the case $k=i=j$ follows immediately from the equations
\begin{align*}
X_{j,j} &= a_{j,j} d_{j,j} + \text{lower terms of type $ad$ or $bc$}  \\
Y_{j,j} &= 2 a_{j,j} c_{j,j} + \text{lower terms of type $ac$ or $bd$}
\end{align*} 
and the fact that $\phi_A(x_{j,j}) =X_{j,j}$ is in the image of $S$. 
\end{proof}

\begin{proposition} \label{prop:extended}
Let $\fn \subset S^{-1}\bC[\ol{B}]$ be the ideal generated by $a_{i,i}-1$, $d_{i,i}-1$, and the remaining variables $($i.e., $a_{i,j}$ for $i \ne j$, $d_{i,j}$ for $i\ne j$, the $b_{i,j}$, and the $c_{i,j})$. Then the extension of $\fm$ to $S^{-1} A'$ is the contraction of $\fn$ to $A'$.
\end{proposition}

\begin{proof}
From the formulas for $\varphi(x_{i,j})$, $\varphi(y_{i,j})$, and $\varphi(z_{i,j})$, one easily sees that the kernel of the homomorphism $A \to \bC[\ol{B}] \to \bC[\ol{B}]/\fn$ is $\fm$. It follows that the kernel of $A \to A' \to A'/\fn^c$ is also $\fm$. Thus the extension of $\fm$ to $A'$ is contained in $\fn^c$. But this extension is maximal, since $S^{-1} \varphi_A$ is an isomorphism, and so we have equality.
\end{proof}

\subsection{The map $\phi_M$ in general}

A {\bf monomial character} of $T$ is a homomorphism $T \to \bC^{\times}$ of the form
\begin{displaymath}
\chi_{\bn} \colon (\alpha_1, \alpha_2, \ldots) \mapsto \alpha_1^{n_1} \alpha_2^{n_2} \cdots
\end{displaymath}
where the $n_i$ are integers and $n_i=0$ for $i \gg 0$. An {\bf admissible representation} of $T$ is a representation $V$ of $T$ that decomposes as a direct sum of monomial characters. We note that if $V$ is an algebraic representation of $\fpe$, $V$ is a representation of its Cartan subalgebra which integrates to an action of $T$; then $V \vert_T$ is an admissible representation of $T$: it suffices to check this for tensor powers of $\bV$ and $\bV[1]$ in which case it is clear.

\begin{proposition} \label{prop:T-ad}
Let $V$ be an admissible representation of $T$. Then $N=S^{-1} (V \otimes \bC[\ol{B}])^T$ is free over $\vert S^{-1} A \vert$.
\end{proposition}

\begin{proof}
It suffices to treat the case when $V$ is one dimensional, say with basis $v$. Suppose that $T$ acts on $V$ through the character $\chi_{\bn}$. Define $p_i$ to be $a_{i,i}$ if $n_i>0$ and $d_{i,i}$ if $n_i<0$ (and 1 if $n_i=0$), and define $p(\bn)=p_1^{n_1} p_2^{n_2} \cdots$. The element $v \otimes p(\bn)$ is $T$-invariant, and an argument similar to the one in the proof of Proposition~\ref{prop:phiA} shows that $S^{-1} (V \otimes \bC[\ol{B}])^T$ is a free $\vert S^{-1} A \vert$-module generated by this element.
\end{proof}

Now let $M$ be an $A$-module, and consider the map
\begin{displaymath}
\varphi_M \colon M \to (M/\fm M \otimes \bC[\ol{B}])^T
\end{displaymath}
The target is naturally a module over the ring $A'$, which is itself an $\vert A \vert$-algebra, and one easily verifies that $\varphi_M$ is a map of $\vert A \vert$-modules.

\begin{lemma} \label{lem:phired}
The reduction of $\varphi_M$ modulo $\fm$ is an isomorphism.
\end{lemma}

\begin{proof}
Let $N=(M/\fm M \otimes \bC[\ol{B}])^T$ and let $\psi \colon N \to M/\fm M$ be the map induced by $\bC[\ol{B}] \to \bC[\ol{B}]/\fn=\bC$, where $\fn$ is as in Proposition~\ref{prop:extended}. Write $\ol{\phi}$ and $\ol{\psi}$ for the mod $\fm$ reductions of $\phi=\phi_M$ and $\psi$. Consider the diagram
\begin{displaymath}
\xymatrix{
M \ar[r] \ar[d] & (M \otimes \bC[\ol{B}])^T \ar[r] \ar[d] &  M \ar[d] \\
M/\fm M \ar[r]^{\ol{\phi}} & N/\fm N \ar[r]^{\ol{\psi}} & M/\fm M }
\end{displaymath}
The top right map is induced by $\bC[\ol{B}] \to \bC[\ol{B}]/\fn$. By definition, the composition of the top row is the action of $1 \in B$ on $M$, and is thus the identity. The diagram is easily seen to commute, and so $\ol{\psi} \circ \ol{\phi}$ is the identity.

Now, let $\{m_i\}$ be a basis of $M/\fm M$ consisting of $T$ weight vectors, where $m_i$ has weight $\bn_i$. Then it follows from the proof of Proposition~\ref{prop:T-ad} that the elements $m_i \otimes p(\bn_i)$ form a basis of $N/\fm N$. Since $p(\bn_i)=1 \pmod{\fn}$, we have $\ol{\psi}(m_i \otimes p(\bn_i))=m_i$. Thus $\ol{\psi}$ takes a basis of $N/\fm N$ to one of $M/\fm M$, and is thus a bijection. Since $\ol{\phi}$ is a right inverse to $\ol{\psi}$, it too is a bijection.
\end{proof}

\begin{lemma} \label{lem:eqker}
The kernel of $\varphi_M$ is $\GL$-stable, and thus an $A$-submodule of $M$.
\end{lemma}

\begin{proof}
By definition, the kernel of $\varphi_M$ consists of those $m \in M$ such that the $B$-submodule of $M$ generated by $m$ is contained in $\fm M$. In other words, $m \in \ker(\varphi_M)$ if and only if $m \in \fm M$ and $am \in \fm M$ for all $a \in \cU(\fb)$. Clearly, $\ker(\varphi_M)$ is $\fb$-stable. It suffices to show that it is also $\pe$-stable, since $\gl=\pe+\fb$ (Lemma~\ref{lem:iwasawa}). Let $Y \in \pe$ and let $m \in \ker(\varphi_M)$, and let us show $Ym \in \ker(\varphi_M)$. Since $m \in \fm M$ and $\fm$ is $\pe$-stable, it follows that $Ym \in \fm M$. Now let $X \in \fb$. We have $XYm=YXm+[X,Y]m$. Now, $Xm \in \fm M$ and so $YXm \in \fm M$. Since $[X,Y]$ belongs to $\gl=\fb+\pe$, we can write it as $X'+Y'$ with $X' \in \fb$ and $Y' \in \pe$. Since $m \in \ker(\varphi_M)$, we have $X'm \in \fm M$ and since $m \in \fm M$ we have $Y'm \in \fm M$. The result follows.
\end{proof}

Since $\fm$ is $\fpe$-stable, so is $S$, and so $\fpe$ acts on $S^{-1}A$. We say that a $\fpe$-equivariant $S^{-1}A$-module is {\bf algebraic} if it is generated, as an $S^{-1}A$-module, by an algebraic $\fpe$-subrepresentation.

\begin{lemma} \label{lem:eqfg}
Suppose that
\begin{displaymath}
0 \to R \to M \to N \to 0
\end{displaymath}
is an exact sequence of algebraic $\fpe$-equivariant $S^{-1}A$-modules such that $M$ is equivariantly finitely generated and $N$ is free as an $\vert S^{-1}A \vert$-module. Then $R$ is also equivariantly finitely generated.
\end{lemma}

\begin{proof}
We first treat the case where $M$ is also $\vert S^{-1}A \vert$-free. Since $R$ is a summand of $M$ as an $\vert S^{-1}A \vert$-module, it follows that $R$ is projective and thus (since $S^{-1}A$ is local) free. Consider the sequence
\begin{displaymath}
0 \to R/\fm R \to M/\fm M \to N/\fm N \to 0
\end{displaymath}
which is exact by the freeness hypothesis on $N$. Since $M$ is finitely generated and algebraic, $M/\fm M$ is a finite length algebraic $\fpe$-representation, and so $R/\fm R$ is as well. Let $V \subset R$ be a finite length algebraic $\fpe$-representation surjecting onto $R/\fm R$. Then Nakayama's lemma shows that $V$ generates $R$ as an $\vert S^{-1}A \vert$-module, which shows that $R$ is equivariantly finitely generated. (Note: we can apply Nakayama without an a priori finiteness condition on $R$ since we know $R$ is free.)

We now treat the general case. Let $V \subset M$ be a finite length algebraic representation surjecting onto $N/\fm N$. Let $N'=S^{-1}A \otimes V$, let $N''$ be the kernel of the surjection $N' \to N$, and let $M'$ be the fiber product of $M$ and $N'$ over $N$. We have the following commutative diagram
\begin{displaymath}
\xymatrix{
&& 0 & 0 \\
0 \ar[r] & R \ar[r] & M \ar[r] \ar[u] & N \ar[r] \ar[u] & 0 \\
0 \ar[r] & R \ar@{=}[u] \ar[r] & M' \ar[u] \ar[r] & N' \ar[r] \ar[u] & 0 \\
& & N'' \ar[u] \ar@{=}[r] & N'' \ar[u] \\
& & 0 \ar[u] & 0 \ar[u] }
\end{displaymath}
The two rows and two columns are exact. Applying the previous paragraph to the right column, we see that $N''$ is equivariantly finitely generated. The middle column now shows that $M'$ is an extension of equivariantly finitely generated modules, and thus equivariantly finitely generated. Now, the surjection $M' \to N'$ splits equivariantly (this is why we introduced $M'$), and so there is an equivariant surjection $M' \to R$, proving that $R$ is equivariantly finitely generated.
\end{proof}

\begin{proposition} \label{prop:phiM}
Let $M$ be an $A$-module. The localization $S^{-1}\varphi_M$ is an isomorphism.
\end{proposition}

\addtocounter{equation}{-1}
\begin{subequations}
\begin{proof}
The assignment $M \mapsto \phi_M$ commutes with filtered colimits, and so it suffices to treat the case where $M$ is finitely generated. Let $R$ be the kernel of $\varphi_M$, which is an $A$-submodule of $M$ by Lemma~\ref{lem:eqfg}, and let $N=S^{-1} (M/\fm M \otimes \bC[\ol{B}])^T$. Since $M/\fm M$ is an admissible representation of $T$ (being an algebraic representation of $\fpe$), Proposition~\ref{prop:T-ad} shows that $N$ is a free $S^{-1} A$-module. By Lemma~\ref{lem:phired}, the map $M/\fm M \to N/\fm N$ is an isomorphism. It follows that $S^{-1} \varphi_{M}$ is a surjection, since it is a surjection mod $\fm$ (which is the Jacobson radical of $S^{-1} A$) and $N$ is free. Since localization is exact, we have an exact sequence of algebraic $\fpe$-equivariant $S^{-1}A$-modules
\begin{equation} \label{eq:phiM}
0 \to S^{-1}R \to S^{-1}M \to N \to 0
\end{equation}
From Lemma~\ref{lem:eqfg}, we conclude that $S^{-1} R$ is equivariantly finitely generated. Let $V \subset R$ be a finite length algebraic representation generating $R$ as an $\vert S^{-1}A \vert$-module, and let $R_0$ be the $A$-submodule of $R$ generated by $V$. Note that $R_0$ is finitely generated as an $A$-module and $S^{-1}R_0=S^{-1}R$. Now, the mod $\fm$ reduction of \eqref{eq:phiM} is exact, by the freeness of $N$, and the reduction of $S^{-1}M \to N$ is an isomorphism. We conclude that $R/\fm R=R_0/\fm R_0=0$. Lemma~\ref{lem:nakayama} thus shows that $0=S^{-1}R_0=S^{-1} R$, and the proposition is proved.
\end{proof}
\end{subequations}

\begin{corollary} 
\label{cor:freeM}
Let $M$ be an $A$-module. Then $S^{-1} M$ is a free $\vert S^{-1} A \vert$-module.
\end{corollary}

\begin{remark}
Proposition~\ref{prop:phiM} is the analog of \cite[Prop.~3.6]{sym2noeth}. The proof of Prop.~3.6 given in \cite{sym2noeth} contains two gaps. First, the justification that $\varphi_M$ is an isomorphism modulo $\fm$ is incomplete. Second, and more seriously, the application of Nakayama's lemma to $R$ is inadequately justified. The above proof fills in these gaps in the present case, and can be easily adapted to fill in the gaps of \cite{sym2noeth}.
\end{remark}

\section{$\Mod_K$ and algebraic representations} \label{s:equiv}

For an $A$-module $M$, define $\wt{\Phi}(M)=M/\fm M$. This is naturally a representation of $\fpe$. The main result of this section is the following theorem:

\begin{theorem} \label{thm:modSA-equiv}
The functor $\wt{\Phi}$ induces an equivalence of categories
\begin{displaymath}
\Phi \colon \Mod_{K} \to \Rep(\fpe).
\end{displaymath}
\end{theorem}

\begin{lemma} \label{lem:free-fiber}
Let $V$ be a polynomial representation of $\GL_{\infty\mid\infty}$. Then $\wt{\Phi}(A \otimes V)$ is isomorphic, as a $\fpe$-representation, to $V$. For any $A$-module $M$, $\wt{\Phi}(M)$ is in $\Rep(\fpe)$.
\end{lemma}

\begin{proof}
The first part is clear. For the second part, pick a surjection $A \otimes V \to M$ of $A$-modules. Since $\wt{\Phi}$ is right exact, there is an induced surjection $V \to\wt{\Phi}(M)$. As any quotient of an algebraic representation is algebraic, we conclude that $\wt{\Phi}(M)$ is algebraic. 
\end{proof}

\begin{lemma} \label{lem:phiexact}
The functor $\wt{\Phi}$ is exact and kills $\Mod_A^{\tors}$.
\end{lemma}

\begin{proof}
Exactness follows from Corollary~\ref{cor:freeM}. Let $M$ be a finitely generated torsion $A$-module. Then $\fa M=0$ for some non-zero ideal $I$ of $A$. As $\fa+\fm=A$ by Corollary~\ref{cor:ann}, we see $M=\fm M$, and so $\wt{\Phi}(M)=M/\fm M=0$. Thus $\wt{\Phi}$ kills finitely generated torsion modules. Since $\wt{\Phi}$ commutes with colimits, it thus kills all torsion $A$-modules.
\end{proof}

Lemma~\ref{lem:free-fiber} shows that $\wt{\Phi}$ takes values in $\Rep(\fpe)$. Lemma~\ref{lem:phiexact} shows that $\wt{\Phi}$ factors uniquely as $\Phi \circ \rT$, where $\rT \colon \Mod_A \to \Mod_K$ is the localization functor, and $\Phi \colon \Mod_K \to \Rep(\fpe)$ is an exact functor. We have thus defined $\Phi$. In the remainder of this section, we prove that $\Phi$ is an equivalence.

\begin{lemma} \label{lem:faithful}
$\Phi$ is faithful.
\end{lemma}

\begin{proof}
Let $f \colon M \to N$ be a map of $A$-modules such that the induced map $\ol{f} \colon M/\fm M \to N/\fm N$ vanishes. The square
\begin{displaymath}
\xymatrix{
M \ar[r]^-{\varphi_M} \ar[d]_f & (M/\fm M \otimes \bC[\ol{B}])^T \ar[d]^{\ol{f} \otimes 1} \\
N' \ar[r]^-{\varphi_{N'}} & (N'/\fm N' \otimes \bC[\ol{B}])^T}
\end{displaymath}
commutes. Since $\varphi_M$ and $\varphi_{N'}$ are isomorphisms after localizing at $S$ (Proposition~\ref{prop:phiM}), we see that the induced map $f \colon S^{-1} M \to S^{-1} N$ is 0, and so $\rT(f)=0$. We have thus shown that if $f$ is any morphism in $\Mod_A$ such that $\wt{\Phi}(f)=0$ then $\rT(f)=0$. Since every morphism in $\Mod_K$ has the form $\rT(f)$ for some morphism $f$ in $\Mod_A$, it follows that $\Phi$ is faithful.
\end{proof}

We now begin the proof of fullness. Let $M$ and $N$ be torsion-free $A$-modules and let $\ol{f} \colon M/\fm M \to N/\fm N$ be a map of $\fpe$-representations. In what follows, a bar denotes reduction mod $\fm$. We write $U$ for the unipotent radical of $B$ and $C$ for the maximal torus, so that $B=CU$. In the notation of \S \ref{ss:subgroups}, $U$ consists of matrices in $B$ where $a$ and $d$ are strictly upper-triangular, while $C$ consists of matrices where $a$ and $d$ are diagonal and $b$ and $c$ vanish.

\begin{lemma} \label{lem:infin}
Let $m \in M$ and let $n \in N$. Let $H \in \{U,C,B\}$ and let $\fh$ be its Lie algebra. Then $\ol{hn}=\ol{f}(\ol{hm})$ as algebraic functions $H \to N / \fm N$ if and only if $\ol{an}=\ol{f}(\ol{am})$ as elements of $N$, for all $a \in \cU(\fh)$. 
\end{lemma}

\begin{proof}
We first prove the result for $H=U$. Let $R$ be a commutative super $\bC$-algebra. We treat elements of $\fh=\fu$ as matrices in the usual way. If $X$ is a super degree~0 element of $\fu \otimes R$, the exponential $\exp(X)=\sum_{n \ge 0} \frac{X^n}{n!}$ is a finite sum and defines an element of $U(R)$, and the map $\exp \colon {}_0(\fu \otimes R) \to U(R)$ is a bijection of sets. Furthermore, if $v$ is a vector in an algebraic representation of $\GL$ then $X^n v=0$ for $n \gg 0$ and $\exp(X) v$ is equal to the finite sum $\sum_{n \ge 0} \frac{X^n}{n!} v$, where $X^n \in \cU(\fh) \otimes R$.

Suppose now that $\ol{an}=\ol{f}(\ol{am})$ for all $a \in \cU(\fu)$. Taking $a=\sum_{n \ge 0} \frac{X^n}{n!}$, we see that $\ol{hn}=\ol{f}(\ol{hm})$ for $h=\exp(X) \in U(R)$. Since every element of $U(R)$ has this form, we conclude that $\ol{hn}=\ol{f}(\ol{hm})$ as functions $H \to N/\fm N$. The reverse direction is similar.

We now treat the case $H=C$. Any algebraic representation of $\GL$ breaks up as a sum of weight spaces for $C$, and the result follows by decomposing $m$ and $n$. Indeed, suppose $\ol{an}=\ol{f}(\ol{am})$ for all $a \in \cU(\fc)$, and write $m=\sum m_i$ and $n=\sum n_i$ where $m_i$ and $n_i$ have weight $\chi_i$. We have $am=\sum \chi_i(a) m_i$ for $a \in \cU(\fc)$, and similarly for $n$. We thus see that $\sum \chi_i(a) \ol{n}_i=\sum \chi_i(a) \ol{f}(\ol{m_i})$ for all $a \in \cU(\fc)$. We conclude that $\ol{n_i}=\ol{f}(\ol{m}_i)$ holds for all $i$ (this uses the fact that characters are linearly independent on $\cU(\fc)$, which requires characteristic~0). Since $S$ acts on $n_i$ and $m_i$ through the same character, it follows that $\ol{hn_i}=\ol{f}(\ol{hm_i})$ for all $h \in S$, and so, summing over $i$, we conclude $\ol{hn}=\ol{f}(\ol{hm})$.

The case $H=B$ follows from the previous two cases, since $B=CU$.
\end{proof}

The diagram in Lemma~\ref{lem:faithful} allows us to define a map $f \colon S^{-1} M \to S^{-1} N$ which is $\vert S^{-1}A \vert$-linear. The map $f$ is characterized by the following lemma.

\begin{lemma} \label{lem:fcrit}
Let $m \in M$ and $n \in N$. Then the following are equivalent:
\begin{enumerate}
\item $n=f(m)$
\item $\ol{hn}=\ol{f}(\ol{hm})$ as functions $H \to N / \fm N$.
\item $\ol{an}=\ol{f}(\ol{am})$ for all $a \in \cU(\fb)$.
\end{enumerate}
\end{lemma}

\begin{proof}
By definition, $\varphi_M(x)$ is the function $B \to M/\fm M$ given by $b \mapsto \ol{bx}$, and so (a) and (b) are equivalent by definition. Lemma~\ref{lem:infin} (with $H=B$) gives the equivalence of (b) and (c).
\end{proof}

\begin{lemma} \label{lem:beq}
Suppose $m \in M$, $n \in N$ and $n=f(m)$. Then $Xn=f(Xm)$ for all $X \in \fb$.
\end{lemma}

\begin{proof}
By Lemma~\ref{lem:fcrit}, we must show $\ol{aXn}=\ol{f}(\ol{aXm})$ for all $a \in \cU(\fb)$. But $aX \in \cU(\fb)$ since $X \in \fb$, and so the identity holds by Lemma~\ref{lem:fcrit}.
\end{proof}

\begin{lemma} \label{lem:peeq}
Suppose $m \in M$, $n \in N$ and $n=f(m)$. Then $Yn=f(Ym)$ for all $Y \in \pe$.
\end{lemma}

\begin{proof}
For $a \in \cU(\fb)$, let $S(a)$ be the following statement:
\begin{quote}
For every $m \in M$ and $n \in N$ and $Y \in \pe$ such that $n=f(m)$ we have $\ol{aYn}=\ol{f}(\ol{aYm})$.
\end{quote}
The statement $S(1)$ holds. Indeed, if $n=f(m)$ then $\ol{n}=\ol{f}(\ol{m})$ and so $\ol{Yn}=Y \ol{f}(\ol{m})=\ol{f}(\ol{Ym})$ since $\ol{f}$ is $\pe$-equivariant. Now suppose $S(a)$ holds, and let us prove $S(aX)$ for $X \in \fb$. Write $[X,Y]=X'+Y'$ with $X' \in \fb$ and $Y' \in \pe$. Then
\begin{align*}
\ol{f}(\ol{aXYm})
&= \ol{f}(\ol{aYXm})+\ol{f}(\ol{aX'm})+\ol{f}(\ol{aY'm}) \\
&= \ol{aYXn}+\ol{aX'n}+\ol{aY'n} \\
&= \ol{aXYn}
\end{align*}
The first line and third lines are clear. Let us explain the second. By Lemma~\ref{lem:beq}, $f(Xm)=Xn$. Thus $\ol{f}(\ol{aYXm})=\ol{aYXn}$ by $S(a)$. We have $\ol{f}(\ol{aX'm})=\ol{aX'n}$ by Lemma~\ref{lem:fcrit}. And we have $\ol{f}(\ol{aY'm})=\ol{aY'n}$ by $S(a)$. We have thus shown that if $S(a)$ holds then $S(aX)$ holds for all $X \in \fb$. It follows that $S(a)$ holds for all $a \in \cU(\fb)$, which (by Lemma~\ref{lem:fcrit}) proves the lemma.
\end{proof}

\begin{lemma} \label{lem:S-lifting}
There exists an $A$-submodule $M'$ of $M$ such that $S^{-1}M'=S^{-1}M$ and for which $f \colon M' \to N$ is a map of $A$-modules.
\end{lemma}

\begin{proof}
Let $M'=M \cap f^{-1}(N)$. Since $f$ is $\vert A \vert$-linear, $M'$ is a $\vert A \vert$-submodule of $M$. Furthermore, for every $m \in M$ there exists $s \in S$ such that $sf(m) \in N$, and so $sm \in M'$. Thus $S^{-1}M'=S^{-1}M$. Finally, it follows from Lemmas~\ref{lem:beq} and~\ref{lem:peeq} that $M'$ is $\gl$-stable and $f$ is $\gl$-equivariant on $M'$, and so the lemma follows.
\end{proof}

\begin{lemma}
The functor $\Phi$ is full.
\end{lemma}

\begin{proof}
Let $\ol{f} \colon M/\fm M \to N/\fm N$ be a given map of $\fpe$-representations. From Lemma~\ref{lem:S-lifting}, we obtain a map $f \colon M' \to N$ of $A$-modules, where $M'$ is an $A$-submodule of $M$ with $S^{-1}M'=S^{-1}M$. Since $S^{-1}(M/M')=0$, it follows that $M/M'$ is torsion, and so the inclusion $M' \to M$ becomes an isomorphism in $\Mod_K$. Thus $f$ defines a map $M \to N$ in $\Mod_K$, and it induces $\ol{f}$ after applying $\Phi$. (Reason: applying $\Phi$ is just reducing modulo $\fm$, and $f$ modulo $\fm$ is $\ol{f}$ by Lemma~\ref{lem:fcrit}.)
\end{proof}

\begin{lemma}
$\Phi$ is essentially surjective.
\end{lemma}

\begin{proof}
Since $\Phi$ is full and compatible with direct limits, it suffices to show that all finitely generated objects of $\Rep(\fpe)$ are in the essential image of $\Phi$. Thus let $M$ be such an object. By Proposition~\ref{prop:pe-inj}(c), we can realize $M$ as the kernel of a map $f \colon I \to J$, where $I$ and $J$ are injective objects of $\Rep(\fpe)$. By Proposition~\ref{prop:pe-inj}(b), every injective object of $\Rep(\fpe)$ is the restriction to $\fpe$ of a polynomial representation of $\GL(\bV)$. Thus, by Lemma~\ref{lem:free-fiber}, $I=\Phi(M)$ and $J=\Phi(N)$ for some $M$ and $N$ in $\Mod_{K}$, and (by fullness) $f=\Phi(f')$ for some $f' \colon M \to N$ in $\Mod_{K}$. The exactness of $\Phi$ shows that $M \cong \Phi(\ker(f'))$, and so $\Phi$ is essentially surjective.
\end{proof}

\section{Proof of the main theorem} \label{s:proof}

In this section, we use the ideas from \cite{sym2noeth} to finish the proof that $A$ is noetherian. Let $\bW$ be another copy of $\bV$ with an action of a separate $\GL_{\infty|\infty}$. The algebra $\Sym(\bV \otimes \bW[1])$ has a natural $\GL(\bV) \times \GL(\bW)$ action which turns it into a bivariate twisted skew-commutative algebra.

\begin{proposition}
$\Sym(\bV \otimes \bW[1])$ is noetherian.   
\end{proposition}

\begin{proof}
If we apply transpose duality (Remark~\ref{rmk:transpose}) with respect to the $\GL(\bW)$-action, then we see that the category of modules over $\Sym(\bV \otimes \bW[1])$ is equivalent to the category of modules over $\Sym(\bV \otimes \bW)$. The latter is noetherian by \cite[Theorem 1.2]{sym2noeth}, and so the same holds for $\Sym(\bV \otimes \bW[1])$.
\end{proof}

Recall from \cite[\S 2.3]{sym2noeth} that a polynomial representation $V$ is {\bf essentially bounded} if there exist integers $r$ and $s$ such that for any $\bS_\lambda$ appearing in $V$ we have $\lambda_r \le s$.

\begin{proposition} \label{prop:A/I-noeth}
If $I$ is a nonzero ideal of $A$, then $A/I$ is essentially bounded, and in particular, noetherian.
\end{proposition}

\begin{proof}
It follows from Corollary~\ref{cor:ess-bound} that $A/\fp_n$ is essentially bounded for all $n$, so the same is true for $A/I$ by Corollary~\ref{cor:ann}. The second part follows from \cite[Proposition 2.4]{sym2noeth}.
\end{proof}

We will need the following fact about the rectangular partitions:

\begin{lemma} \label{lem:LR-rect}
We have $\bS_{n \times k} \subset \bS_\lambda \otimes \bS_\mu$ if and only if $\lambda$ and $\mu$ are complementary shapes in the $n \times k$ rectangle, i.e., $\ell(\lambda), \ell(\mu) \le n$ and $\lambda_i + \mu_{n+1-i} = k$ for $i=1,\dots,n$.
\end{lemma}

\begin{proof}
This is a statement about Littlewood--Richardson numbers which is more transparent in the context of Schubert calculus, see \cite[\S 9.4, eqn. (11)]{fulton}.
\end{proof}

Recall the notion of (FT) from \cite[\S 4.2]{sym2noeth}: if $B$ is a twisted (skew-)commutative algebra, and $M$ is a $B$-module, then $M$ satisfies (FT) over $B$ if $\Tor^B_i(M, \bC)$ is a finite length $\GL(\bV)$-module for all $i \ge 0$. While the definitions and results were stated only in the commutative case, they work perfectly well in the skew-commutative case.

\begin{lemma}
If $I$ is a nonzero ideal of $A$, then $A/I$ satisfies {\rm (FT)} over $A$.
\end{lemma}

\begin{proof}
We will follow the proof of \cite[Lemma 4.6]{sym2noeth}\footnote{There is a typo in the published version: $J_\lambda$ should be the ideal generated by $\bS_{2\lambda} \otimes \bS_{2\lambda}$ when $B = \Sym(\Sym^2 \bC^\infty)$, or $\bS_{(2\lambda)^\dagger} \otimes \bS_{(2\lambda)^\dagger}$ when $B = \Sym(\bigwedge^2 \bC^\infty)$.}. By Corollary~\ref{cor:ann}, there exists $n$ such that $I \supseteq \fp_n$ (recall that $\fp_n$ is the ideal generated by $\bS_{n \times (n+1)}$). Let $J_n \subset \Sym(\bV \otimes \bW[1])$ be the ideal generated by $\bS_{n \times (n+1)}(\bV) \otimes \bS_{(n+1) \times n}(\bW)$. Let $\wt{C}$ be the tca $\Sym(\bV \otimes \bV[1])$ with the diagonal action of $\GL(\bV)$. Then there is a surjection of tca's $\phi \colon \wt{C} \to A$ induced by the natural map $\bV^{\otimes 2} \to \Sym^2(\bV)$. By Corollary~\ref{cor:ess-bound}, we have $\phi(J_n) \subseteq \fp_n \subseteq I$.

We claim that $\phi(J_n) \ne 0$. To see this, write $\wt{C} = \Sym(\Sym^2(\bV)[1]) \otimes \Sym(\bigwedge^2(\bV)[1])$. It suffices to show that $\bS_{n \times (n+1)}$ is not in the ideal generated by $\bigwedge^2(\bV)[1]$, and for that, we will show that if $\bS_{n \times (n+1)} \subset \bS_\lambda \otimes \bS_\mu$ where $\bS_\lambda \subset \Sym(\Sym^2(\bV)[1])$ and $\bS_\mu \subset \Sym(\bigwedge^2(\bV)[1])$, then $\mu = \emptyset$. We prove this by induction on $n$; when $n=1$, this is clear. By Lemma~\ref{lem:LR-rect}, this happens if and only if $\ell(\lambda), \ell(\mu) \le n$ and $\lambda_i + \mu_{n+1-i} = n+1$ for $i=1,\dots,n$, i.e., $\lambda$ and $\mu$ are complementary shapes inside of the $n \times (n+1)$ rectangle. Now, $\lambda \in Q_1$ (see \S\ref{sec:fm}) which implies that $\lambda_1 = \ell(\lambda) + 1$. Furthermore, $\mu^\dagger \in Q_1$, which implies that $\mu_1 = \ell(\mu) - 1$. If $\ell(\lambda) < n$, then we must have $\mu_1 = n+1$ and $\ell(\mu) = n$, which is a contradiction, so $\ell(\lambda) = n$. In this case, remove the first row and column from $\lambda$ to get a new shape $\lambda'$ with complementary shape $\mu$ inside of the $(n-1) \times n$ rectangle. By Lemma~\ref{lem:LR-rect}, we have $\bS_{(n-1) \times n} \subset \bS_{\lambda'} \otimes \bS_\mu$. So by induction on $n$, we conclude that $\lambda' = (n-1) \times n$ and hence $\mu = \emptyset$. We conclude that $\phi(J_n) \supset \bS_{n \times (n+1)}$ and hence $\phi(J_n) \ne 0$.

Now we can finish using the arguments from \cite[Lemma 4.6]{sym2noeth}. Some final points: $\wt{C} / J_n$ is (FT) over $\wt{C}$ since we can apply transpose duality (Remark~\ref{rmk:transpose}) to \cite[Lemma 4.5]{sym2noeth}, and $A / \phi(J_n)$ is noetherian by Proposition~\ref{prop:A/I-noeth}.
\end{proof}

\begin{corollary} \label{cor:tors-mod-FT}
If $M$ is a finitely generated $A$-module with nonzero annihilator, then $M$ satisfies {\rm (FT)} over $A$.
\end{corollary}

\begin{proof}
The proof follows as in \cite[Proposition 4.3]{sym2noeth}.
\end{proof}

Recall that $\rT \colon \Mod_A \to \Mod_K$ is the localization functor, where $\Mod_K = \Mod_A / \Mod_A^{\tors}$. Let $\rS \colon \Mod_K \to \Mod_A$ be the section functor, which is the right adjoint to localization. An object $M \in \Mod_A$ is {\bf saturated} if $\ext^i_A(N, M) = 0$ for $i=0,1$ and all $N \in \Mod_A^{\tors}$. This is equivalent to the unit of the adjunction $M \to \rS(\rT(M))$ being an isomorphism.

\begin{proposition} \label{prop:free-sat}
Given a finite length representation $V$ of $\GL(\bV)$, we have $\rS(\rT(V \otimes A)) = V \otimes A$, i.e., $V \otimes A$ is saturated.
\end{proposition}

\begin{proof}
Pick $N \in \Mod_A^\tors$. It is clear that $\hom_A(N, V \otimes A) = 0$ since no submodule of $V \otimes A$ is annihilated by a nonzero ideal. Now we show that $\ext^1_A(N, V \otimes A) = 0$. First we assume that $N$ is finitely generated. 

Pick a minimal $A$-free resolution $\bF_\bullet \to N \to 0$ of $N$. Since $N$ is (FT) over $A$ (Corollary~\ref{cor:tors-mod-FT}), each $\bF_i$ is finitely generated. Pick $n$ larger than the number of rows of any minimal generator $\bS_\lambda$ of $\bF_i$ for $i \le 2$. Then the natural map 
\[
\hom_A(\bF_i, V \otimes A) \to \hom_{A(\bC^n)}(\bF_i(\bC^n), (V \otimes A)(\bC^n))^{\GL_n(\bC)}
\]
is an isomorphism for $i \le 2$. Each of these Hom groups is an algebraic representation of $\GL_n(\bC)$, so taking invariants is exact. We conclude that the map
\[
\ext^1_A(N, V \otimes A) \to \ext^1_{A(\bC^n)}(\bF_i(\bC^n), (V \otimes A)(\bC^n))^{\GL_n(\bC)}
\]
is also an isomorphism. Now note that $(V \otimes A)(\bC^n) = V(\bC^n) \otimes A(\bC^n)$ is a finite rank free module over an exterior algebra in a finite number of variables. Since the exterior algebra in finitely many variables is self-injective, we conclude that $(V \otimes A)(\bC^n)$ is an injective $A(\bC^n)$-module. In particular, the desired $\ext^1$ group vanishes.

Now suppose $N$ is not finitely generated. Then $N$ can be written as a countable colimit $N = \varinjlim N_{\alpha}$ of finitely generated submodules. Since $\Hom(-, V\otimes A)$ commutes with colimit, we have a spectral sequence with $E_2$ term $ \varprojlim^i \Ext^j(N_{\alpha}, V\otimes A)$ converging to $\Ext^{i+j}(N, V\otimes A)$. Since the colimit is countable we have $\varprojlim^i = 0$ for $i > 1$. By the previous paragraph, $ \varprojlim^1 \Ext^0(N_{\alpha}, V\otimes A) = \varprojlim^0 \Ext^1(N_{\alpha}, V\otimes A) = 0$. Thus we have $\Ext^{1}(N, V\otimes A) = 0$, as required.
\end{proof}

\begin{proposition}
If $M$ is a finite length $S^{-1} A$-module, then $\rS(M)$ satisfies {\rm (FT)} over $A$.
\end{proposition}

\begin{proof}
The proof is the same as \cite[Proposition 4.8]{sym2noeth} using the results from \S\ref{sec:Perep} and Proposition~\ref{prop:free-sat}.
\end{proof}

\begin{theorem}
The tca $A = \Sym(\Sym^2(\bV)[1])$ is noetherian.
\end{theorem}

\begin{proof}
The proof is the same as \cite[Theorem 4.9]{sym2noeth}.
\end{proof}

\begin{remark}
To get the same result for $\bigwedge^\bullet(\bigwedge^2)$, we can apply transpose duality to $A$. Alternatively, we could follow the proof outlined above; the point is that an odd skew-symmetric bilinear form on $\bV$ is exactly the same thing as an odd symmetric bilinear form on $\bV$.
\end{remark}

\end{document}